\documentclass[9pt,a4paper]{amsart}

\usepackage{amssymb}
\usepackage{hyperref}
\usepackage{orcidlink}

\theoremstyle{plain}
\newtheorem{theorem}{Theorem}[section]
\newtheorem{lemma}[theorem]{Lemma}
\newtheorem{proposition}[theorem]{Proposition}
\newtheorem{corollary}[theorem]{Corollary}

\theoremstyle{definition}

\theoremstyle{remark}
\newtheorem{remark}[theorem]{Remark}

\newcommand{\Ps}{\mathbb{P}}

\def\bin #1#2 {\left( \matrix { #1 \cr #2 \cr } \right) }

\begin{document}

\title[The genus of a curve on Calabi-Yau threefolds]
{The genus of a curve on Calabi-Yau threefolds}

\author{Vincenzo Di Gennaro }
\address{Universit\`a di Roma \lq\lq Tor Vergata\rq\rq, Dipartimento di Matematica,
Via della Ricerca Scientifica, 00133 Roma, Italy.}
\email{digennar@mat.uniroma2.it}

\author{Antonio Rapagnetta}
\address{Universit\`a di Roma \lq\lq Tor Vergata\rq\rq, Dipartimento di Matematica,
Via della Ricerca Scientifica, 00133 Roma, Italy.}
\email{rapagnet@mat.uniroma2.it}

\author{Pietro Sabatino\, \orcidlink{0009-0005-9454-5910}}
\address{Independent Researcher, Rome, Italy.}
\email{pietro.sabatino74@gmail.com}

\abstract
Let $X$ be a smooth projective Calabi-Yau variety of dimension $3$, with Picard number one. Let $n$ be the self-intersection
of the ample generator of the N\'eron-Severi group of $X$.
Let $\mathsf{GV}_{g,d}$ be the Gopakumar-Vafa invariant of $X$ of genus $g$ and degree $d$. It is conjectured that
{\it if $g>\frac{d^{2}}{2n}+\frac{d}{2}+1$,  then $\mathsf{GV}_{g,d}=0$}. In this paper,
we prove the implication when $d\gg n$ and $X\subset \mathbb P^r$ is a complete intersection in  projective space.
The vanishing follows by proving that if $C$ is
a closed subscheme of $X$ of dimension $1$, arithmetic genus $p_a(C)$,
and degree $d\gg n$, then
$p_a(C) \leq \frac{d^2}{2n}+\frac{d}{2}+1$.
Actually, we obtain a somewhat more general result.
In particular, we prove that $C$ is an extremal curve if and only if
$C$ is the complete intersection on $X$ with a hyperplane $H$ and a hypersurface $F$, namely $C=X\cap H\cap F$.
When $X\subset\mathbb P^r$ is not necessarily a complete intersection, assuming that  the Picard group of $X$ is
generated by the hyperplane section class, we obtain the coarse bound
$p_a(C)\leq \frac{d^2}{2n}+\frac{d}{2}+\left(\frac{11}{24}n+4\right)$ for $d\gg n$.



\endabstract
\maketitle

\section{Introduction}
Let $X$ be a smooth projective Calabi-Yau variety (i.e. $\mathcal O_X(K_X)\cong \mathcal O_X$)
of dimension $3$, with Picard number one. Let $n$ be the self-intersection
of the ample generator of the N\'eron-Severi group of $X$. Let $\mathsf{GV}_{g,d}$ be the Gopakumar-Vafa invariant of $X$ of genus $g$ and degree $d$.
The following conjecture is stated in \cite[p. 2, Conjecture 1.1]{Liu}.

\smallskip
{\bf{Conjecture}} (Castelnuovo bound conjecture).
$$
g> \frac{d^2}{2n}+\frac{d}{2}+1 \quad \implies \quad \mathsf{GV}_{g,d}=0.
$$

\medskip
The conjecture has been proven for the
quintic in $\mathbb P^4$ and the complete intersection of type $(2,4)$ in $\mathbb P^5$ (see
\cite[p. 2, Conjecture 1.1]{Liu} and the references therein).

\smallskip
In this paper,
we prove the implication when $d\gg n$ and $X$ is a complete intersection in  the complex projective space
$\mathbb P^r_{\mathbf C}=\mathbb P^r$.

\smallskip
\begin{theorem}\label{GV}
Let $X\subseteq\mathbb P^r$ be a smooth projective  complete intersection Calabi-Yau threefold,  of degree $n$.
Let $\mathsf{GV}_{g,d}$ be the Gopakumar-Vafa invariant of $X$ of genus $g$ and degree $d$. Assume that $d\gg n$. Then we have:
$$
g> \frac{d^2}{2n}+\frac{d}{2}+1 \quad \implies \quad \mathsf{GV}_{g,d}=0.
$$
\end{theorem}

\medskip
By the same argument in \cite[p. 36, Proof of Theorem 6.4]{Liu}, in order to prove Theorem \ref{GV} it suffices to prove
the following:

\smallskip
\begin{theorem}\label{nunico}
Let $X\subseteq\mathbb P^r$ be a smooth projective  complete intersection Calabi-Yau threefold, of degree $n$.
Let $C$ be  a closed subscheme of $X$ of dimension $1$, arithmetic genus $p_a(C)$,
and of degree $d\gg n$. Then we have:
\begin{equation}\label{Cbound}
p_a(C) \leq \frac{d^2}{2n}+\frac{d}{2}+1.
\end{equation}
Moreover, equality holds  if and only if
$C$ is the complete intersection on $X$ with a hyperplane $H$ and a hypersurface $F$, namely $C=X\cap H\cap F$.
\end{theorem}

\smallskip
In turn, Theorem \ref{nunico} is a consequence of the following Theorem \ref{unico} and Theorem \ref{2,2,2,2}.
In Remark $(vi)$ below, we make explicit the condition $d\gg n$ in all our claims.

\smallskip
\begin{theorem}\label{unico}
Let $X\subseteq\mathbb P^r$ be a smooth projective Calabi-Yau threefold, of degree $n$.
Assume the following conditions are satisfied:

\smallskip
(1) the Picard group of $X$ is generated by the hyperplane section class;

\smallskip
(2) $h^1(X,\mathcal O_X)=0$;

\smallskip
(3) $h^0(X,\mathcal O_X(1))\geq \frac{n}{2}$.


\smallskip\noindent
Let $C$ be  a closed subscheme of $X$ of dimension $1$, arithmetic genus $p_a(C)$,
and of degree $d\gg n$. Divide $$d-1=mn+\epsilon,\quad  0\leq \epsilon\leq n-1.$$
Set
\begin{equation}\label{psidiepsilon}
\psi(\epsilon)=\frac{1+\epsilon}{2n}(n-1-\epsilon)-2.
\end{equation}
Then we have:
\begin{equation*}
\psi(\epsilon)\leq 0\quad \implies \quad p_a(C) \leq \frac{d^2}{2n}+\frac{d}{2}+1.
\end{equation*}
Moreover, assume that (regardless of the sign of $\psi(\epsilon)$)
\begin{equation*}
p_a(C) =\frac{d^2}{2n}+\frac{d}{2}+1.
\end{equation*}
We have: if $\epsilon=n-1$, then  $C$ is the complete intersection on $X$ with a hyperplane $H$ and a hypersurface $F$ of degree $\frac{d}{n}$,
namely $C=X\cap H\cap F$; if $n$ is square-free, then $\epsilon=n-1$; if $n< 16$, then $\epsilon=n-1$;
if $n=16$, then either $\epsilon=n-1=15$ or $\epsilon=\frac{n-2}{2}=7$.
\end{theorem}

\smallskip
\begin{theorem}\label{2,2,2,2}
Let $X\subseteq\mathbb P^7$ be a smooth projective complete intersection  Calabi-Yau threefold, of degree $n=16$.
Let $C$ be  a closed subscheme of $X$ of dimension $1$, arithmetic genus $p_a(C)$, and
of degree $d\gg n$, with $d-1=mn+7$. Then we have:
\begin{equation*}
p_a(C) < \frac{d^2}{2n}+\frac{d}{2}+1.
\end{equation*}
\end{theorem}

\medskip
\begin{remark}\label{osservazioni} $(i)$
The complete intersection Calabi-Yau threefolds are: $X_5\subset \mathbb P^4$ with $n=5$,
$X_{2,4}\subset \mathbb P^5$ with $n=8$, $X_{3,3}\subset \mathbb P^5$ with $n=9$,
$X_{2,2,3}\subset \mathbb P^6$ with $n=12$, $X_{2,2,2,2}\subset \mathbb P^7$ with $n=16$.

\smallskip
$(ii)$ One has: {\it $\psi(\epsilon)\leq 0$ for every $0\leq \epsilon\leq n-1$ if and only if $n\leq 16$}.


\smallskip
$(iii)$ As far as we know, the characterization of the extremal case does not appear in the literature.

\smallskip
$(iv)$ For every $n$,
if $\epsilon\in\{0,\,1,\,2,\,3,\,n-5,\,n-4,\,n-3,\,n-2,\,n-1\}$, then $\psi(\epsilon)\leq 0$. Moreover, if $n\leq 25$, then $\psi(4)\leq 0$;
if $n\leq 18$, then $\psi(5)\leq 0$; etc.

\smallskip
$(v)$ If $n$ is square-free and $\frac{d^2}{2n}+\frac{d}{2}+1$ is an integer, then $d$ is a multiple of $n$, i.e. $\epsilon=n-1$.

\smallskip
$(vi)$  In proving Theorem \ref{nunico}, Theorem \ref{unico} and Theorem \ref{2,2,2,2}, we need the numerical assumption
$d\gg n$ in order to apply \cite[Corollary 3.3.]{DRS}, which
allows us to assume that the curve $C$ lies on a hyperplane section
of $X$ (see below). Therefore, it suffices to assume $d>d_0:=6n^2M(n)^2$, where
$$
M(n)=\frac{n^2}{2}\left[\frac{2}{3}(n+1)+4+2\left(\frac{4}{3}n(n+1)\right)^6\right]^2.
$$
Moreover, in order to prove Theorem \ref{GV} as a consequence of Theorem \ref{nunico}, we need to assume  $d>nd_0$ in
Theorem \ref{GV}.
One may use  the same argument in \cite[p. 36, Proof of Theorem 6.4]{Liu}.
We omit the details.

\end{remark}

\medskip
In short, the outline of the proof is as follows.
By a result from our recent work
\cite[Corollary 3.3.]{DRS},
we may assume that $C$ lies on a hyperplane
section $S$ of $X$. Given the hypothesis on the Picard group, the surface $S$ is necessarily integral.
Thus, the problem reduces to estimating
the genus of a curve on an integral surface. This can be done using classical methods:
the genus formula in \cite[Proposition 3.9]{DF},
Clifford's theorem for integral curves \cite[Proposition 1.5]{Fujita}, Castelnuovo's base-point free pencil trick \cite[p. 447]{E},
Hopf's theorem \cite[Proposition 1.3]{LS}, and its limiting case, namely  \cite[Theorem 33]{BSZ},
which we will refer to as \lq\lq Vosper's theorem\rq\rq.

\smallskip
When $X$ is not necessarily a complete intersection, our argument easily produces the following
coarse bound.  It should be compared with \cite[Corollary 1.4]{DRS} and with
\cite[Theorem 1.5]{Li}.

\begin{corollary}\label{general}
Let $X\subseteq\mathbb P^r$ be a smooth projective Calabi-Yau threefold, of degree $n$.
Assume  that $h^1(X,\mathcal O_X)=0$, and that  the Picard group of $X$ is
generated by the hyperplane section class.
Let $C$ be  a closed subscheme of $X$ of dimension $1$
and of degree $d\gg n$. Then we have:
$$
p_a(C)
\leq \frac{d^2}{2n}+\frac{d}{2}+\left(\frac{11}{24}n+4\right).
$$
\end{corollary}


\section{Notations and preliminary results}

In this section we establish the notation to be used subsequently.
We prove some preliminary results, and recall certain results known in the literature.

\smallskip
$(i)$ Let $X\subset \mathbb P^r$ be as in the claim of Theorem \ref{unico}.
 Set
\begin{equation}\label{acca0}
h=h^0(X,\mathcal O_X(1)).
\end{equation}
We may assume $r=h-1$.
Let $S\subset \mathbb P^{r-1}$ be a hyperplane section of $X$. By the hypothesis on the Picard group of $X$, $S$ is an integral surface of degree $n$.
Let $p_a(S)=\chi(\mathcal O_S)-1$ be the arithmetic genus of $S$.
Let $\Sigma\subset \mathbb P^{r-2}$ be a general hyperplane section of $S$.
Then $\Sigma$ is an integral Gorenstein curve of degree $n$, arithmetic genus
$p_a(\Sigma)=n+1$, with dualizing sheaf given by $\omega_{\Sigma}=\mathcal O_{\Sigma}(2)$.

\begin{lemma}\label{acca}
We have:
$$
p_a(S)=h-1,\quad  h^1(S,\mathcal O_S)=0,\quad  h^0(S,\mathcal O_{S}(1))=h-1, \quad h^0(\Sigma,\mathcal O_{\Sigma}(1))=h-2.
$$
\end{lemma}

\begin{proof}
Taking into account that $\mathcal O_X(K_X)\cong \mathcal O_X$ and that $h^1(X,\mathcal O_X)=0$, the Kodaira Vanishing Theorem, and the Serre Duality Theorem, from
the natural exact sequence
$$
0\to \mathcal O_X(-1)\to \mathcal O_X \to \mathcal O_S\to 0
$$
we get
$$
p_a(S)=\chi(\mathcal O_S)-1=\chi(\mathcal O_X)-\chi(\mathcal O_X(-1))-1=h^3(X,\mathcal O_X(-1))-1=h-1.
$$
Moreover, from the long exact sequence induced in cohomology, since $h^1(X,\mathcal O_X)=0$ and
$h^1(X,\mathcal O_X(-1))=0$,  we get
$$
h^1(S,\mathcal O_S)=0.
$$
Similarly, twisting the previous sequence by $\mathcal O_X(1)$,
we get
$h^0(S,\mathcal O_{S}(1))=h-1$.
Since  $h^1(S,\mathcal O_S)=0$ and  $h^0(S,\mathcal O_{S}(1))=h-1$, from
the natural exact sequence
$$
0\to \mathcal O_S\to \mathcal O_S(1) \to \mathcal O_{\Sigma}(1)\to 0
$$
we get $h^0(\Sigma,\mathcal O_{\Sigma}(1))=h-2$.
\end{proof}

\medskip
$(ii)$ Let $C$ be a closed subscheme of $X\subset \mathbb P^r$ of dimension $1$, degree $d$, and arithmetic genus $p_a(C)$, as in Theorem \ref{unico}.
By \cite[Corollary 3.3.]{DRS},  in order to prove the bound (\ref{Cbound}) on $p_a(C)$, we may assume that $C$ lies on a hyperplane section $S$ of $X$,
$C\subset S\subset \mathbb P^{r-1}$.
Let $\Gamma\subset\Sigma\subset\mathbb P^{r-2}$ be a general hyperplane section of $C$. We denote by $h_C$, $h_{\Gamma}$, $h_S$ the
corresponding Hilbert functions. Divide
\begin{equation}\label{divide}
d-1=mn+\epsilon, \quad 0\leq \epsilon\leq n-1.
\end{equation}
Recall that we are assuming that $d\gg n$. In the following lemma (compare with
\cite[Proposition 3.9 and Remark 3.5]{DF}) we denote by $\pi_X$ the linear sectional genus of $X$, i.e.
$\pi_X=p_a(\Sigma)=n+1$.

\begin{lemma}\label{RCMP}
We have:
\begin{equation}\label{f1}
p_a(C)= \binom{m}{2}n+m(\epsilon+\pi_X)-p_a(S) +\sum_{i=m+1}^{+\infty}d-h_{\Gamma}(i)-\sum_{i=m+1}^{+\infty}\delta_i,
\end{equation}
where
$$
\delta_i
=\dim\ker\left[H^1(\mathbb P^{r-1},\mathcal I_{C,\mathbb P^{r-1}}(i-1))\to H^1(\mathbb P^{r-1},\mathcal I_{C,\mathbb P^{r-1}}(i))\right].
$$
In particular
\begin{equation*}
p_a(C)\leq \binom{m}{2}n+m(\epsilon+\pi_X)-p_a(S) +\sum_{i=m+1}^{+\infty}d-h_{\Gamma}(i).
\end{equation*}
\end{lemma}

\begin{proof}
Since for $t\gg 0$
we have $h_C(t)=1-p_a(C)+dt$, we may write
\begin{equation}\label{sslprop}
p_a(C)=dt+1-h_C(t)=\sum_{i=1}^{t}d-\Delta
h_{C}(i)=\sum_{i=1}^{+\infty}d-\Delta h_{C}(i)
\end{equation}
$$
=\sum_{i=1}^{m}d-\Delta h_{C}(i)+\sum_{i=m+1}^{+\infty}d-\Delta
h_{C}(i).
$$
On the other hand, for degree reasons,  we have $h_C(i)=h_S(i)$ for
every $i\leq m$, and therefore we have
$$
\sum_{i=1}^{m}d-\Delta h_{C}(i)=md+1-h_C(m)=md+1-h_S(m).
$$
Since $d\gg n$, the Hilbert function  $h_S(m)$ coincides with
the Hilbert polynomial $p_S(m)$ of $S$ at level $m$ (compare with \cite[Lemma 3.8]{DF}), i.e.
$$
h_S(m)=p_S(m)={\binom{m+1}{2}}n+m(1-\pi_X)+1+p_a(S).
$$
It follows that
$$
\sum_{i=1}^{m}d-\Delta h_{C}(i)=md+1-h_S(m)
$$
$$
=md+1-\left[{\binom{m+1}{2}}n+m(1-\pi_X)+1+p_a(S)\right]={\binom{m}{2}}n+m(\epsilon+\pi_X)-p_a(S).
$$
Inserting  into (\ref{sslprop}) we obtain:
$$
p_a(C)= \binom{m}{2}n+m(\epsilon+\pi_X)-p_a(S) +\sum_{i=m+1}^{+\infty}d-\Delta h_{C}(i).
$$
Now we observe that
$$
\sum_{i=m+1}^{+\infty} d-\Delta
h_{C}(i)=\sum_{i=m+1}^{+\infty}d-h_{\Gamma}(i)
-\sum_{i=m+1}^{+\infty}\Delta h_C(i)-h_{\Gamma}(i).
$$

\smallskip
It remains to prove that for every $i$
$$\Delta h_C(i)-h_{\Gamma}(i)=\delta_i.$$

\smallskip
From the exact sequence
$$
0\to \mathcal I_{C,\mathbb P^{r-1}}(i-1)\to \mathcal I_{C,\mathbb P^{r-1}}(i) \to \mathcal I_{\Gamma,\mathbb P^{r-2}}(i)\to 0
$$
we get
\begin{equation}\label{snake}
\delta_i=\dim{\text{Im}}\left(H^0(\mathbb P^{r-2},\mathcal I_{\Gamma,\mathbb P^{r-2}}(i)) \to
H^1(\mathbb P^{r-1},\mathcal I_{C,\mathbb P^{r-1}}(i-1))\right)
\end{equation}
$$
=h^0(\mathbb P^{r-2},\mathcal I_{\Gamma,\mathbb P^{r-2}}(i))-
\dim \ker\left(H^0(\mathbb P^{r-2},\mathcal I_{\Gamma,\mathbb P^{r-2}}(i)) \to
H^1(\mathbb P^{r-1},\mathcal I_{C,\mathbb P^{r-1}}(i-1))\right)
$$
$$
=h^0(\mathbb P^{r-2},\mathcal I_{\Gamma,\mathbb P^{r-2}}(i))-h^0(\mathbb P^{r-1},\mathcal I_{C,\mathbb P^{r-1}}(i))+
h^0(\mathbb P^{r-1},\mathcal I_{C,\mathbb P^{r-1}}(i-1)).
$$
On the other hand
$$
\Delta h_C(i)=h^0(\mathbb P^{r-1},\mathcal O_{\mathbb P^{r-1}}(i))-h^0(\mathbb P^{r-1},\mathcal I_{C,\mathbb P^{r-1}}(i))
$$
$$
-h^0(\mathbb P^{r-1},\mathcal O_{\mathbb P^{r-1}}(i-1))+h^0(\mathbb P^{r-1},\mathcal I_{C,\mathbb P^{r-1}}(i-1))
$$
$$
=
h^0(\mathbb P^{r-2},\mathcal O_{\mathbb P^{r-2}}(i))-h^0(\mathbb P^{r-1},\mathcal I_{C,\mathbb P^{r-1}}(i))
+h^0(\mathbb P^{r-1},\mathcal I_{C,\mathbb P^{r-1}}(i-1)).
$$
And
$$
h_{\Gamma}(i)=h^0(\mathbb P^{r-2},\mathcal O_{\mathbb P^{r-2}}(i))-h^0(\mathbb P^{r-2},\mathcal I_{\Gamma,\mathbb P^{r-2}}(i)).
$$
Therefore, we have
$$
\Delta h_C(i)-h_{\Gamma}(i)=h^0(\mathbb P^{r-2},\mathcal O_{\mathbb P^{r-2}}(i))-h^0(\mathbb P^{r-1},\mathcal I_{C,\mathbb P^{r-1}}(i))
+h^0(\mathbb P^{r-1},\mathcal I_{C,\mathbb P^{r-1}}(i-1))
$$
$$
-h^0(\mathbb P^{r-2},\mathcal O_{\mathbb P^{r-2}}(i))+h^0(\mathbb P^{r-2},\mathcal I_{\Gamma,\mathbb P^{r-2}}(i))
$$
$$
=h^0(\mathbb P^{r-2},\mathcal I_{\Gamma,\mathbb P^{r-2}}(i))-h^0(\mathbb P^{r-1},\mathcal I_{C,\mathbb P^{r-1}}(i))+
h^0(\mathbb P^{r-1},\mathcal I_{C,\mathbb P^{r-1}}(i-1)),
$$
which is equal to $\delta_i$ by (\ref{snake}).
\end{proof}

\smallskip
Inserting  $m=(d-1-\epsilon)/n$ in (\ref{f1}), and taking into account that $\pi_X=n+1$, we get:
\begin{corollary}\label{stime}
$$
p_a(C)= \frac{d^2}{2n}+\frac{d}{2}+\rho-\sum_{i=m+1}^{+\infty}\delta_i,
$$
where
\begin{equation}\label{rho}
\rho=-\frac{1+\epsilon}{2n}(n+1+\epsilon)-p_a(S)+\sum_{i=m+1}^{+\infty}d-h_{\Gamma}(i).
\end{equation}
In particular, we have:
$$
p_a(C)\leq \frac{d^2}{2n}+\frac{d}{2}+\rho.
$$
\end{corollary}

\smallskip
$(iii)$
Our strategy is to prove that $\rho\leq 1$. A first estimate on $\rho$ follows using
Clifford's theorem for  {a quasi-invertible sheaf} on $\Sigma$
(compare with \cite[p. 166, Proposition 1.5]{Fujita}, and \cite[Proof of Lemma 3.3]{DF}).
We keep the previous notation.

\begin{proposition}\label{m+1}
For every $j\geq 1$ we have:
\begin{equation*}
d-h_{\Gamma}(m+j)\leq \max\left\{0, \,
\left\lfloor\frac{2\pi_X-(jn-1-\epsilon)}{2}\right\rfloor\right\}.
\end{equation*}
\end{proposition}

\begin{proof}
Let $\mathcal I\subset \mathcal O_{\Sigma}$
be the ideal sheaf of $\Gamma$ in $\Sigma$.
For every $j\geq 1$ consider  the natural exact
sequence:
\begin{equation}\label{previouses}
0\to\mathcal I(m+j)\to
\mathcal{O}_{\Sigma}(m+j)\to\mathcal{O}_{\Gamma}\to 0.
\end{equation}
Since $\Sigma$ is integral, the sheaf $\mathcal I(m+j)$ is a torsion-free coherent sheaf of rank one
(i.e. {a quasi-invertible sheaf}).
Following \cite[p. 166]{Fujita}, define:

\smallskip
$\bullet$
$\deg \mathcal I(m+j):=\chi(\mathcal I(m+j))-1+\pi_X$,

\smallskip
$\bullet$
$\Delta \mathcal I(m+j):=1+\deg \mathcal I(m+j)-h^0(\mathcal I(m+j))=\pi_X-h^1(\mathcal I(m+j)) $.

\smallskip
Observe that for every $j$ one has  (compare with (\ref{divide})):
\begin{equation}\label{geqp}
\deg \mathcal I(m+j)=\chi(\mathcal I(m+j))-1+\pi_X=\chi(\mathcal O_{\Sigma}(m+j))-d-1+\pi_X=jn-1-\epsilon.
\end{equation}
In particular, we have:
\begin{equation}\label{geq}
j\geq 1\quad\implies\quad \deg \mathcal I(m+j)=jn-1-\epsilon\geq 0.
\end{equation}
Since $d\gg n$, by \cite[Theorem, $(i)$, p. 492]{GLP},
we know that the natural map
$$
H^0(\Ps^{r-2},
\mathcal{O}_{\Ps^{r-2}}(m+j))\to H^0(\Sigma,
\mathcal{O}_{\Sigma}(m+j))
$$
is surjective for every $j\geq 1$. Moreover, since $d\gg n$, in view of the Castelnuovo bound
for $\pi_X$ \cite[Theorem (3.7), p. 87]{EH}, we also have:
$$
h^1(\Sigma,\mathcal
O_{\Sigma}(m+j))=0
$$
for every $j\geq 1$.
Hence,
from the previous exact sequence (\ref{previouses}), we get:
$$
h_{\Gamma}(m+j)=h^0(\Sigma, \mathcal{O}_{\Sigma}(m+j))-h^0(\Sigma,
\mathcal I(m+j))
$$
$$
= 1-\pi_X+(m+j)n-h^0(\Sigma, \mathcal I(m+j)),
$$
i.e.
\begin{equation}\label{ex}
d-h_{\Gamma}(m+j)=h^0(\Sigma,
\mathcal I(m+j))-jn+\pi_X+\epsilon.
\end{equation}
Observe that, combining (\ref{geq}) and (\ref{ex}), and taking into account the definition of the degree given above,
i.e. $\deg \mathcal I(m+j):=\chi(\mathcal I(m+j))-1+\pi_X$,
we get
\begin{equation}\label{exx}
d-h_{\Gamma}(m+j)=h^1(\Sigma,
\mathcal I(m+j)).
\end{equation}

\smallskip $\bullet$
If $h^1(\Sigma, \mathcal I(m+j))=0$ then $h_{\Gamma}(m+j)=d$, and we are done.

\smallskip $\bullet$
Otherwise, $h^1(\Sigma, \mathcal I(m+j))>0$.

\smallskip
In this case we claim that Clifford's
inequality holds true, i.e.
\begin{equation}\label{claim}
h^0(\Sigma, \mathcal I(m+j))-1\leq
\frac{\deg \mathcal I(m+j)}{2}.
\end{equation}
We will prove (\ref{claim}) in a moment. Combining (\ref{claim}) with (\ref{ex}) and (\ref{geq}) we get
$$
d-h_{\Gamma}(m+j)=h^0(\Sigma,
\mathcal I(m+j))-jn+\pi_X+\epsilon
$$
$$
\leq 1+\frac{\deg \mathcal I(m+j)}{2}-jn+\pi_X+\epsilon=
\frac{2\pi_X-(jn-1-\epsilon)}{2}.
$$
Provided we prove (\ref{claim}),  this concludes the proof of
Proposition \ref{m+1}.

\smallskip
It remains to prove (\ref{claim}).
To this purpose, we argue by contradiction.
Suppose not. We have
\begin{equation}\label{contr}
h^0(\Sigma, \mathcal I(m+j))-1>
\frac{\deg \mathcal I(m+j)}{2}.
\end{equation}
From which we get
\begin{equation}\label{cuno}
\deg \mathcal I(m+j)\geq 2\Delta \mathcal I(m+j)+1.
\end{equation}
In fact, we have:
$$
h^0(\Sigma, \mathcal I(m+j))-1>
\frac{\deg \mathcal I(m+j)}{2}\implies \deg \mathcal I(m+j)< 2h^0(\Sigma, \mathcal I(m+j))-2
$$
$$
\implies \deg \mathcal I(m+j)\leq 2h^0(\Sigma, \mathcal I(m+j))-3
$$
$$
\implies \deg \mathcal I(m+j)\geq 2(1+\deg\mathcal I(m+j)-h^0(\Sigma, \mathcal I(m+j)))+1
$$
$$
\implies  \deg \mathcal I(m+j)\geq 2\Delta \mathcal I(m+j)+1.
$$
Moreover, from (\ref{contr}), we also have
\begin{equation}\label{cdue}
2\Delta \mathcal I(m+j)+1\geq -1.
\end{equation}
In fact, from (\ref{contr}) it follows that $h^0(\Sigma, \mathcal I(m+j))>0$ (recall that from (\ref{geq}) we have
$\deg \mathcal I(m+j)\geq 0$). A non-zero section of $H^0(\Sigma, \mathcal I(m+j))$ induces an exact sequence
$$
0\to \mathcal O_{\Sigma}\to \mathcal I(m+j)\to \mathcal T\to 0,
$$
where $\mathcal T$ is a sheaf supported on a finite set, because the section
$0\to \mathcal O_{\Sigma}\to \mathcal I(m+j)$ is an isomorphism at the general point of $\Sigma$.
Taking cohomology we get the exact sequence
$$
H^1(\Sigma,\mathcal O_{\Sigma})\to H^1(\Sigma, \mathcal I(m+j))\to  H^1(\Sigma,\mathcal T)=0.
$$
Therefore,
$$
h^1(\Sigma, \mathcal I(m+j))\leq h^1(\Sigma,\mathcal O_{\Sigma})=\pi_X.
$$
It follows that
$$
\Delta \mathcal I(m+j)=\pi_X-h^1(\Sigma, \mathcal I(m+j))\geq 0,
$$
which implies (\ref{cdue}).

\smallskip
Summing up, in the case $h^1(\Sigma, \mathcal I(m+j))>0$, from (\ref{contr}) we deduce
(\ref{cuno}) and (\ref{cdue}). By \cite[p. 166, Proposition 1.5]{Fujita} the conditions  (\ref{cuno}) and (\ref{cdue})
imply $\Delta \mathcal I(m+j)=\pi_X$. Hence $h^1(\Sigma, \mathcal I(m+j))=0$, a contradiction.

\smallskip
This concludes the proof of Proposition \ref{m+1}.
\end{proof}

Since $\pi_X=n+1$, by Proposition \ref{m+1} we deduce that the sum  appearing
in (\ref{rho})  consists only of the first three terms, i.e.
\begin{corollary}\label{3terms} We have:
$$
\sum_{i=m+1}^{+\infty}d-h_{\Gamma}(i)=(d-h_{\Gamma}(m+1))+(d-h_{\Gamma}(m+2))+(d-h_{\Gamma}(m+3)).
$$
Moreover, $d-h_{\Gamma}(m+3)=0$ for $\epsilon < n-1$, and $d-h_{\Gamma}(m+3)\leq 1$ when $\epsilon=n-1$.
\end{corollary}

\smallskip
Now, set
$$
\mathcal A=\mathcal I(m+1).
$$
Recall that $\mathcal I$ is the ideal sheaf of $\Gamma$ in $\Sigma$. Therefore, $\mathcal A$ is a
torsion-free coherent sheaf of rank one. Observe that if $C$ is a complete intersection on $S$, then
$\mathcal A=\mathcal O_{\Sigma}$.
From (\ref{geq}), (\ref{ex}), and (\ref{exx}) we may write:

\begin{equation}\label{formulae}
\begin{cases}
d-h_{\Gamma}(m+1)=h^0(\Sigma, \mathcal A)+1+\epsilon=h^1(\Sigma, \mathcal A);\\
d-h_{\Gamma}(m+2)=h^0(\Sigma, \mathcal A(1))+1+\epsilon-n=h^1(\Sigma, \mathcal A(1));\\
d-h_{\Gamma}(m+3)=h^0(\Sigma, \mathcal A(2))+1+\epsilon-2n=h^1(\Sigma, \mathcal A(2)).
\end{cases}
\end{equation}
Moreover,
\begin{equation}\label{degree}
\deg\mathcal A=n-1-\epsilon,\quad\quad \deg\mathcal A(1)=2n-1-\epsilon,\quad \quad \deg\mathcal A(2)=3n-1-\epsilon.
\end{equation}

Set:
\begin{equation}\label{setx}
x=h^0(\Sigma, \mathcal A),\quad y=h^1(\Sigma, \mathcal A(1)),\quad z=h^1(\Sigma, \mathcal A(2)).
\end{equation}
From (\ref{psidiepsilon}), (\ref{acca0}), Lemma \ref{acca}, (\ref{rho}), and (\ref{formulae}), we may write:
\begin{equation}\label{nrho}
\rho=\psi(\epsilon)+3+x+y+z-h.
\end{equation}

\begin{lemma}\label{rhoic}
If $\mathcal A=\mathcal O_{\Sigma}$, then $\epsilon=n-1$ and $\rho=1$.
\end{lemma}

\begin{proof}
Since $\omega_{\Sigma}=\mathcal O_{\Sigma}(2)$, if  $\mathcal A=\mathcal O_{\Sigma}$, by duality we have
$$z=h^1(\Sigma, \mathcal A(2))=h^0(\Sigma, \mathcal O_{\Sigma})=1.$$
From (\ref{formulae}) and Corollary \ref{3terms} we get $\epsilon=n-1$. In this case we have
$\psi(n-1)=-2$, $x=1$, $z=1$, and
by Lemma \ref{acca},  $y=h^0(\Sigma, \mathcal O_{\Sigma}(1))=h-2$. From (\ref{nrho}) it follows $\rho=1$.
\end{proof}

\smallskip
$(iv)$ {\it Castelnuovo's base-point free pencil trick}.
Let $\mathcal L$ be an invertible sheaf on $\Sigma$. Fix a subspace $P$ of $H^0(\Sigma, \mathcal L)$ of dimension $2$.
Assume that $P$ generates $\mathcal L$ at every point.
Then this subspace $P$
induces a Koszul exact sequence \cite[p. 447]{E}:
$$
0\to \mathcal L^{-1}\to P\otimes \mathcal O_{\Sigma}\to \mathcal L\to 0.
$$
Let $\mathcal F$ be a coherent sheaf on $\Sigma$.
By tensoring with $\mathcal F$, since $\mathcal L$ is locally free, we obtain the
exact sequence
$$
0\to \mathcal L^{-1}\otimes \mathcal F\to P\otimes \mathcal F\to \mathcal L\otimes \mathcal F\to 0.
$$
We deduce:
\begin{equation*}
2\,h^0(\Sigma,\mathcal F)\leq h^0(\Sigma,\mathcal L\otimes \mathcal F)+h^0(\Sigma, \mathcal L^{-1}\otimes \mathcal F).
\end{equation*}
In the special case $\mathcal L=\mathcal O_{\Sigma}(1)$, we obtain:
\begin{equation}\label{bpfpt}
2\,h^0(\Sigma,\mathcal F)\leq h^0(\Sigma,\mathcal F(1))+h^0(\Sigma,\mathcal F(-1)).
\end{equation}

\smallskip
$(v)$ {\it Hopf's theorem}. Let $K=K(\Sigma)$ be the field of rational functions on $\Sigma$.
Let $U$ and $V$ be finite-dimensional subspaces of $K$ over $\mathbb C$ with positive dimension.
Let $UV$ be the subspace of $K$ generated by the products $uv$, as $u$ ranges over $U$ and $v$ ranges over $V$.
By Hopf's theorem \cite[Proposition 1.3, p. 702]{LS} we know  that
$$
\dim UV\geq \dim U+\dim V-1.
$$

\smallskip
$(vi)$ {\it Vosper's theorem}. In the case where $\dim U\geq 2$ and $\dim V\geq 2$, and  equality holds
$$
\dim UV = \dim U+\dim V-1,
$$
by \cite[Theorem 33, p. 445]{BSZ} we know there exist elements $g,g',a\in K$
such that $U$ and $V$ have bases of the form
$$
\{g,ga,\dots,ga^{\dim U-1}\}, \quad \quad \{g',g'a,\dots,g'a^{\dim V-1}\}.
$$
It follows that, given $\gamma=gg'$, $UV$ has as a basis
$$
\{\gamma, \gamma a,\dots,\gamma a^{\dim UV-1}\}.
$$

\smallskip
$(vii)$
Let us see how  Hopf's theorem and Vosper's theorem apply in our context.

\smallskip
Consider the natural map
$$
\mathcal A\otimes \mathcal A^{\vee}\to \mathcal O_{\Sigma},
$$
where $\mathcal A^{\vee}={ \mathcal Hom}_{\Sigma}(\mathcal A,\mathcal O_{\Sigma})$.
Tensoring by $\mathcal O_{\Sigma}(1)$, we obtain
$$
\mathcal A\otimes_{\mathcal O_{\Sigma}} \left(\mathcal A^{\vee}(1)\right)\to \mathcal O_{\Sigma}(1).
$$
Passing to global sections, we derive a map
\begin{equation}\label{mappa}
\varphi:U\otimes_{\mathbb C} V\to W,
\end{equation}
where
$$U=H^0(\Sigma,\mathcal A),\quad V=H^0(\Sigma,\mathcal A^{\vee}(1)),\quad
W=H^0(\Sigma,\mathcal O_{\Sigma}(1)).$$
Observe that
\begin{equation}\label{y}
H^1(\Sigma,\mathcal A(1))^{\vee}={\text{Hom}}_{\Sigma}(\mathcal A(1),\mathcal \omega_{\Sigma})={\text{Hom}}_{\Sigma}(\mathcal A,\mathcal O_{\Sigma}(1))=
H^0(\Sigma, \mathcal A^{\vee}\otimes\mathcal O_{\Sigma}(1)).
\end{equation}
Therefore, we have (compare with Lemma \ref{acca} and (\ref{setx})):
\begin{equation*}
x=\dim U, \quad  y=\dim V, \quad h-2=\dim W.
\end{equation*}
Since $\mathcal A$ is a torsion-free coherent sheaf of rank one on $\Sigma$, so is
$\mathcal A^{\vee}(1)$. Therefore, { we can view $U$, $V$, and $W$ as vector subspaces of the field of rational
functions $K(\Sigma)$ on the curve $\Sigma$, and $\varphi$ as the multiplication map}. Consequently,
${\text{Im}}\varphi=UV$. By Hopf's theorem, we deduce that
\begin{equation}\label{ht}
x>0,\, y>0\quad \implies \quad x+y-1\leq h-2.
\end{equation}
And by Vosper's theorem,
\begin{equation}\label{vt}
x\geq 2,\, y\geq 2\quad \implies \quad x+y-1< h-2.
\end{equation}
Indeed, if the equality $x+y-1= h-2$ were to hold, then we would have $UV=W$, and by Vosper's Theorem
the linear system $|\mathcal O_{\Sigma}(1)|$, which is very ample, would map $\Sigma$ to a rational normal curve in projective space.
This is not possible, because the genus of $\Sigma$ is $n+1>0$.


\section{The proof of Theorem \ref{unico} in the case $\epsilon=n-1$}

In this section, we will prove Theorem \ref{unico} in the case $\epsilon=n-1$.
More precisely,
we will prove the following slightly more general  result:

\begin{proposition}\label{new-1}
Let $X\subseteq\mathbb P^r$ be a smooth Calabi-Yau variety of
dimension $3$ and degree $n$. Assume that the Picard group of $X$ is
generated by the hyperplane section class. Assume that $h^1(X,\mathcal O_X)=0$ and that
$$
h^0(X,\mathcal O_X(1))\geq \frac{n-1}{2}.
$$
Let $C$ be  a closed subscheme of $X$ of dimension $1$, arithmetic genus $p_a(C)$,
and of degree $d\gg n$, which is a multiple of $n$, i.e. $d=(m+1)n$. Then we have:
\begin{equation*}
p_a(C) \leq \frac{d^2}{2n}+\frac{d}{2}+1.
\end{equation*}
Moreover, if equality holds and $h^0(X,\mathcal O_X(1))\geq  \frac{n}{2}$, then
$C$ is the complete intersection on $X$ with a hyperplane $H$ and a hypersurface $F$, namely $C=X\cap H\cap F$.
\end{proposition}

\begin{proof} We keep the notations stated in the previous section.

Since $\epsilon=n-1$, by (\ref{degree}) we know that $\deg \mathcal A=0$.
Therefore, we have
$$
0\leq h^0(\Sigma, \mathcal A)\leq 1.
$$

\smallskip
Let us distinguish the two cases.

\smallskip
$\bullet$ If $h^0(\Sigma, \mathcal A)=1$, then, given that $\deg\mathcal A=0$, it follows that
$\mathcal A=\mathcal O_{\Sigma}$. But then,
from Lemma \ref{rhoic}, we know that $\rho=1$, and we are done by Corollary \ref{stime}.

\smallskip
Let us examine what happens in this case if $C$ has the maximal genus $$p_a(C) = \frac{d^2}{2n}+\frac{d}{2}+1.$$
In this case, by Corollary \ref{stime}, it follows that $\delta_{m+1}=0$. On the other hand, since
$\mathcal A=\mathcal O_{\Sigma}$, we have $h_{\Gamma}(m+1)=h_{\Sigma}(m+1)-1$.
Therefore,
there exists a hypersurface of degree $m+1$ that contains
$\Gamma$ but does not contain $\Sigma$. Since $\delta_{m+1}=0$, this hypersurface
must extend to a hypersurface $F$ of degree $m+1$ containing $C$ but not $S$.
Since $S$ is Cohen-Macaulay, by degree considerations, we  have that $C=F\cap S=F\cap H\cap X$
is a complete intersection on $X$.

\smallskip $\bullet$
To complete the proof of Proposition \ref{new-1}, it suffices to show that, in the remaining case $h^0(\Sigma, \mathcal A)=0$,
we have $\rho\leq 1$, and $\rho<1$ if $h^0(X,\mathcal O_X(1))\geq \frac{n}{2}$.

\smallskip
We begin by proving that in this case $d-h_{\Gamma}(m+3)=0$. By (\ref{formulae}), this is equivalent to proving that
$h^1(\Sigma, \mathcal A(2))=0$. By the definition of the dualizing sheaf, we have:
$$
h^1(\Sigma, \mathcal A(2))=\dim {\text{Hom}} (\mathcal A(2),\omega_{\Sigma}).
$$
On the other hand, $\omega_{\Sigma}=\mathcal O_{\Sigma}(2)$. Therefore,
$$
h^1(\Sigma, \mathcal A(2))=\dim {\text{Hom}} (\mathcal A(2),\mathcal O_{\Sigma}(2))=\dim {\text{Hom}} (\mathcal A,\mathcal O_{\Sigma}).
$$
We conclude by observing that ${\text{Hom}} (\mathcal A,\mathcal O_{\Sigma})=0$. Indeed, otherwise, as before,
we would have $\mathcal A=\mathcal O_{\Sigma}$, which contradicts the fact that $h^0(\Sigma, \mathcal A)=0$.

\smallskip
Thus, by (\ref{nrho}) and (\ref{formulae}), we have
$$
\rho=1+y-h=1+h^0(\Sigma, \mathcal A(1))-h^0(X,\mathcal O_X(1)).
$$

\smallskip$\bullet$
From the base-point free pencil trick (\ref{bpfpt}) with $\mathcal F=\mathcal A(1)$, noting that
$h^0(\Sigma, \mathcal A(2))=n$ because $h^1(\Sigma, \mathcal A(2))=0$ (see (\ref{formulae})), we obtain
$$
h^0(\Sigma, \mathcal A(1))\leq \left\lfloor\frac{n}{2}\right\rfloor.
$$
Therefore, we have
$$
h^0(X, \mathcal O_X(1))\geq \frac{n-1}{2}\quad \implies \quad \rho\leq 1.
$$

\smallskip$\bullet$
To conclude, it suffices to show that
$$
h^0(X, \mathcal O_X(1))\geq \frac{n}{2}\quad \implies \quad \rho< 1.
$$
The same computation as above proves this implication when $n$ is odd. We may therefore assume $n$ is even.
In this case, it will suffice to show that we must have
$$
h^0(\Sigma, \mathcal A(1))<\frac{n}{2}.
$$
We proceed by contradiction. That is, we assume for the sake of contradiction that
for every pencil $$P\subset H^0(\Sigma,\mathcal O_{\Sigma}(1))$$ without base points,
the natural map given by the base-point free pencil trick
$$
\Psi_P:P\otimes H^0(\Sigma,\mathcal A(1))\to H^0(\Sigma,\mathcal A(2))
$$
is an isomorphism.

\smallskip
1) For simplicity, we introduce the following notation:
$$
W=H^0(\Sigma,\mathcal O_{\Sigma}(1)), \quad E=H^0(\Sigma,\mathcal A(1)), \quad F=H^0(\Sigma,\mathcal A(2)).
$$
We have
$$
\dim E=\frac{n}{2}, \quad \dim F=n.
$$
Let
$
\mathbb G=\mathbb G(2,W)
$
denote the Grassmannian of pencils in $W$. Let
$
\mathcal S\subseteq W\otimes \mathcal O_{\mathbb G}
$
be the locally free sheaf associated with the tautological bundle.

\smallskip
2) The multiplication map
$
W\otimes E\to F
$
induces a bundle map:
$$
\Psi: \mathcal S\otimes E\to F\otimes \mathcal O_{\mathbb G}.
$$
On the fiber corresponding to a pencil $P$, even if it has a base point, this map coincides with the map $\Psi_P$ mentioned above:
$$
\Psi_P:P\otimes E\to F.
$$
Since the bundles $\mathcal S\otimes E$ and $F\otimes \mathcal O_{\mathbb G}$ have the same rank $n$, we can
consider the determinant of the map $\Psi$:
$$
\det\Psi\in H^0(\mathbb G, (\det (\mathcal S\otimes E))^{\vee}\otimes \det F)=
H^0(\mathbb G, (\det (\mathcal S\otimes E))^{\vee}).
$$
On the other hand,
$$
\det (\mathcal S\otimes E)=(\det\mathcal S)^{\frac{n}{2}}\otimes (\det E)^{2}.
$$
Passing to the duals, we obtain
$$
(\det (\mathcal S\otimes E))^{\vee}=((\det \mathcal S)^{\vee})^{\frac{n}{2}}=\mathcal O_{\mathbb G}\left(\frac{n}{2}\right).
$$
Thus,
$$
\det\Psi\in H^0\left(\mathbb G,\mathcal O_{\mathbb G}\left(\frac{n}{2}\right)\right).
$$
Based on the previous observation, $\det\Psi$ is a non-zero section. Let  $\operatorname{div}(\det\Psi)$ be
the  divisor of zeros of $\det\Psi$. We have:
$$\operatorname{div}(\det\Psi)\in \left|\mathcal O_{\mathbb G}\left(\frac{n}{2}\right)\right|.$$

\smallskip
3) Now consider
$$
D=\{P\in\mathbb G\,:\,{\text{$P$ has a base point on $\Sigma$}}\}.
$$
Since
$$
P\in D\quad\iff\quad \mathbb P\left(P^{\perp}\right)\cap \Sigma\neq \emptyset,
$$
$D$ is the Chow hypersurface of $\Sigma$ in $\mathbb P\left(W^{\vee}\right)$. Since $\Sigma$ is an integral curve of degree
$n$, $D$ is an irreducible divisor of degree $n$ \cite[Proposition 2.2, p. 99]{GKZ}:
$$
D\in \left|\mathcal O_{\mathbb G}\left(n\right)\right|.
$$
Since $D$ is irreducible and
$$
\operatorname{Supp}\left(\operatorname{div}(\det\Psi)\right)\subseteq D,
$$
we have reached a contradiction based on degrees.

\smallskip
This concludes the proof of Proposition \ref{new-1}.
\end{proof}


\section{The proof of Theorem \ref{unico} in the case $\epsilon<n-1$}

We now proceed to prove Theorem \ref{unico} for $\epsilon < n-1$.
We will show that
$\rho < 1$, except when $n=16$ and $\epsilon=7$, in which case we will only be able to prove that $\rho \leq 1$.
We keep all the notations stated in Section $2$.

\medskip
\begin{proof}[Proof of Theorem \ref{unico} in the case $\epsilon<n-1$]
By Corollary \ref{3terms} and (\ref{nrho}), we have (compare with (\ref{acca0}) and (\ref{setx}))
$$
\rho=\psi(\epsilon)+3+x+y-h.
$$

\smallskip$\bullet$
First case: $x=0$.

\smallskip
By Corollary \ref{3terms} and (\ref{formulae}), we have
$$
h^0(\Sigma,\mathcal A(2))=2n-1-\epsilon.
$$
Applying the base-point free pencil trick (\ref{bpfpt}) to $\mathcal F=\mathcal A(1)$,
we deduce that
$$
y=h^0(\Sigma,\mathcal A(1))+1+\epsilon-n\leq \frac{2n-1-\epsilon}{2}+1+\epsilon-n=
\frac{1+\epsilon}{2}.
$$
It follows:
$$
\rho\leq \psi(\epsilon)+3+\frac{1+\epsilon}{2}-h=1+\frac{1+\epsilon}{2n}(2n-1-\epsilon)-h
$$
$$
=1+\frac{n}{2}-\frac{1}{2n}(n-1-\epsilon)^2-h<1+\frac{n}{2}-h,
$$
and so $\rho<1$ because $h\geq \frac{n}{2}$.

\smallskip$\bullet$
Second case: $y=0$.

\smallskip
Applying the base-point free pencil trick (\ref{bpfpt}) to $\mathcal F=\mathcal A$, and noting that
$$h^0(\Sigma,\mathcal A(-1))=0$$ because $\deg \mathcal A(-1)=-1-\epsilon<0$ (see (\ref{geqp})), we obtain
(compare with (\ref{formulae}) and (\ref{setx}))
$$
x\leq \frac{1}{2}h^0(\Sigma,\mathcal A(1))=\frac{n-1-\epsilon}{2}.
$$
It follows:
$$
\rho\leq \psi(\epsilon)+3+\frac{n-1-\epsilon}{2}-h=1+\frac{n+1+\epsilon}{2n}(n-1-\epsilon)-h
$$
$$
=1+\frac{n^2-(\epsilon+1)^2}{2n}-h<1+\frac{n}{2}-h,
$$
and so $\rho<1$ because $h\geq \frac{n}{2}$.

\smallskip$\bullet$
Third case: $x>0$ and $y>0$. In this case, by Hopf's theorem (\ref{ht}), we have:
$$
x+y\leq h-1.
$$

\smallskip
We shall examine four  subcases.

\smallskip
(a) $x+y=h-1$, $x\geq 2$, and $y\geq 2$.

\smallskip
This case is not possible by Vosper's theorem (\ref{vt}).

\smallskip
(b) $x+y=h-1$, $x=1$, and $y=h-2$.

\smallskip
This case is not possible.

\smallskip
In fact, a non-zero global section of $\mathcal A$ induces an exact sequence
$$
0\to\mathcal O_{\Sigma}\to \mathcal A\to\mathcal Q\to 0,
$$
where $\mathcal Q$ is a coherent sheaf supported in dimension zero, with $$\deg\mathcal Q=\deg \mathcal A=n-1-\epsilon>0.$$
Dualizing yields
the exact sequence
$$
0\to \mathcal Q^{\vee}\to \mathcal A^{\vee}\to \mathcal O_{\Sigma}\to
{\mathcal Ext}^1_{\Sigma}(\mathcal Q,\mathcal O_{\Sigma})\to {\mathcal Ext}^1_{\Sigma}(\mathcal A,\mathcal O_{\Sigma}).
$$
Since $\Sigma$ is an integral Gorenstein curve and $\mathcal A$ is a torsion-free coherent sheaf of rank one, we have:
$$
\mathcal Q^{\vee}={\mathcal Ext}^1(\mathcal A,\mathcal O_{\Sigma})=0.
$$
Therefore, the sequence reduces to:
$$
0 \to \mathcal A^{\vee}\to \mathcal O_{\Sigma}\to
{\mathcal Ext}^1_{\Sigma}(\mathcal Q,\mathcal O_{\Sigma})\to 0,
$$
with ${\mathcal Ext}^1_{\Sigma}(\mathcal Q,\mathcal O_{\Sigma})$ supported in dimension zero. Moreover,
$$
\deg {\mathcal Ext}^1_{\Sigma}(\mathcal Q,\mathcal O_{\Sigma})=\deg \mathcal Q>0
$$
because $\deg \mathcal A^{\vee}=-\deg\mathcal A$.
Tensoring by $\mathcal O_{\Sigma}(1)$ yields:
$$
0 \to \mathcal A^{\vee}(1)\to \mathcal O_{\Sigma}(1)\to
{\mathcal Ext}^1_{\Sigma}(\mathcal Q,\mathcal O_{\Sigma})\to 0.
$$
Since the degree of ${\mathcal Ext}^1_{\Sigma}(\mathcal Q,\mathcal O_{\Sigma})$ is $>0$, the map on global sections induced by $\mathcal O_{\Sigma}(1)\to
{\mathcal Ext}^1_{\Sigma}(\mathcal Q,\mathcal O_{\Sigma})$ cannot be zero. Thus (compare with (\ref{y})):
$$
y=h^1(\Sigma,\mathcal A(1))=h^0(\Sigma, \mathcal A^{\vee}(1))\leq h^0(\Sigma, \mathcal O_{\Sigma}(1))-1=h-3.
$$
This contradicts the hypothesis $y=h-2$. Therefore, subcase (b) cannot occur.

\smallskip
(c) $x+y=h-1$ and $x=h-2$ and $y=1$.

\smallskip
This case is not possible.

\smallskip
Recall that $y=h^1(\Sigma,\mathcal A(1))=h^0(\Sigma, \mathcal A^{\vee}(1))=1$.
A non-zero global section of $\mathcal A^{\vee}(1)={\text{Hom}}(\mathcal A,\mathcal O_{\Sigma}(1))$ induces an exact sequence
$$
0\to\mathcal A\to \mathcal O_{\Sigma}(1)\to\mathcal Q\to 0,
$$
where $\mathcal Q$ is a coherent sheaf supported in dimension zero, with
$$\deg\mathcal Q=n-\deg\mathcal A=1+\epsilon>0.$$
Since the map $\mathcal O_{\Sigma}(1)\to\mathcal Q$ cannot be zero on global sections,
it follows that $x\leq h-3$, which is impossible.

\smallskip
(d) $x+y\leq h-2$.

\smallskip
Since by hypothesis $\psi(\epsilon)\leq 0$, we have:
$$
\rho=\psi(\epsilon)+3+x+y-h\leq 3+x+y-h\leq 1.
$$

\smallskip
To conclude, we show that if $n< 16$ (or $n=16$ and $\epsilon\neq 7$), then
$$
\psi(\epsilon)<0.
$$

\smallskip
To this end, we notice that the function $\psi(\epsilon)$
attains an absolute maximum at $\epsilon=\frac{n-2}{2}$. Therefore,
$$
\psi(\epsilon)\leq \psi\left(\frac{n-2}{2}\right)=\frac{1}{2n}\left[\frac{n}{4}(n-16)\right].
$$
Thus, $\psi(\epsilon)<0$
for $n<16$. And this holds true even if $n=16$, provided $\epsilon\neq 7$.
\end{proof}


\section{The proof of Theorem \ref{2,2,2,2}}

\begin{proof}[Proof of Theorem \ref{2,2,2,2}]

Compared to the proof of the Theorem \ref{unico}, the proof of Theorem \ref{2,2,2,2} contains no new arguments.
Therefore, we will proceed more quickly. We keep all the notations stated before.

For simplicity, let
$$
\mathcal B=\mathcal A^{\vee}(1).
$$

\smallskip
We have:
$$
n=16,\quad h=8,\quad  \epsilon=7, \quad \rho=x+y-5, \quad \deg\mathcal A=\deg\mathcal B=8.
$$
We need to show that $x+y\leq 5$.
By Theorem \ref{unico} we already know that $\rho\leq 1$. We will assume for the sake of contradiction that $\rho=1$, i.e., that
$x+y=6$.
As in the proof of Theorem \ref{unico}, we can assume $x>0$ and $y>0$. Thus, the cases to be ruled out are:
$$
(x,y)=(1,5), \quad (x,y)=(2,4),\quad (x,y)=(3,3),\quad (x,y)=(4,2),\quad (x,y)=(5,1).
$$

\medskip$\bullet$
{\it Step 1. The cases $(x,y)=(1,5)$ and $(x,y)=(5,1)$ are not possible.}

\smallskip
First, suppose $x=h^0(\Sigma,\mathcal A)=1$ and $y=h^0(\Sigma,\mathcal B)=5$.

A non-zero section of $\mathcal A$ induces
\[
0\longrightarrow\mathcal O_\Sigma\longrightarrow\mathcal A\longrightarrow \mathcal Q\longrightarrow0,
\qquad \operatorname{deg}\mathcal Q=\deg\mathcal A=8.
\]
By dualizing and tensoring by $\mathcal O_\Sigma(1)$, we obtain
\[
0\longrightarrow\mathcal B\longrightarrow\mathcal O_\Sigma(1)\longrightarrow \mathcal T\longrightarrow0,
\qquad \operatorname{deg}\mathcal T=8.
\]
At the level of global sections, $V=H^0(\Sigma,\mathcal B)$ is identified with a hyperplane in $W=H^0(\Sigma,\mathcal O_{\Sigma}(1))$. Thus, there exists a unique point
$q\in\mathbb P\left(W^\vee\right)=\mathbb P^5$
such that $V$ is the space of linear forms vanishing at $q$.

If $q\notin\Sigma$, the sections of $V$ generate $\mathcal O_\Sigma(1)$ everywhere. Since they originate from $H^0(\Sigma,\mathcal B)$,
the map to $\mathcal T$ is zero, implying $\mathcal T=0$. This contradicts the fact that $\operatorname{deg}\mathcal T=8$.

If $q\in\Sigma$, the subsheaf generated by $V$ is $\mathcal I_{q/\Sigma}(1)$. Therefore,
$\mathcal I_{q/\Sigma}(1)\subseteq\mathcal B$,
and $\mathcal T$ is a quotient of
$\mathcal O_\Sigma(1)/\mathcal I_{q/\Sigma}(1)\simeq k(q)$.
It follows that $\operatorname{deg}(\mathcal T)\leq1$, in contrast with $\deg\mathcal T=8$.  Thus, the case
$(x,y)=(1,5)$ is impossible.

\smallskip
Now suppose $x=h^0(\Sigma,\mathcal A)=5,\quad y=h^0(\Sigma,\mathcal B)=1$.

\smallskip
A non-zero section of $\mathcal B$ induces
\[
0\longrightarrow\mathcal O_\Sigma\longrightarrow\mathcal B\longrightarrow \mathcal Q'\longrightarrow0,
\qquad \operatorname{deg} Q'=\deg\mathcal B=8.
\]
By dualizing and tensoring with $\mathcal O_\Sigma(1)$, we obtain
\[
0\longrightarrow\mathcal B^\vee(1)\longrightarrow\mathcal O_\Sigma(1)
\longrightarrow \mathcal T'\longrightarrow0,
\qquad \operatorname{deg} \mathcal T'=8.
\]
Since $\mathcal A$ is a rank-one reflexive sheaf on the Gorenstein curve $\Sigma$, we have
$\mathcal B^\vee(1)=\bigl(\mathcal A^\vee(1)\bigr)^\vee(1)\simeq\mathcal A$. Therefore, the
previous sequence becomes
\[
0\longrightarrow\mathcal A\longrightarrow\mathcal O_\Sigma(1)\longrightarrow \mathcal T'\longrightarrow0.
\]
Now, $H^0(\Sigma,\mathcal A)$ is a hyperplane of $W$. The same argument as before
shows that the quotient $\mathcal T'$ is either zero or has degree at most one. This contradicts $\operatorname{deg} \mathcal T'=8$. Therefore,
the case $(x,y)=(5,1)$ is also impossible.

\medskip
{\it Having concluded this step, we may assume
$$
x\geq 2,\quad y\geq 2.
$$
This new hypothesis will be applied in the next step}.

\medskip$\bullet$
{\it Step 2. If $\rho=1$, then the map $\varphi: U\otimes V\to W$ in (\ref{mappa}) must be surjective.}

\smallskip
By Hopf's theorem, there are only two possibilities:
$$
\dim UV=6, \quad {\text{or}}\quad \dim UV=5.
$$
If $\dim UV=5$, the minimal Hopf equality $\dim UV=\dim U+\dim V-1$ holds.
By Vosper's theorem, we know there exist $a,b,f\in K(\Sigma)$ such that
$$
UV=(ab)\cdot {\operatorname {Span}}(1,f,f^2,f^3,f^4)\subset W=H^0(\Sigma,\mathcal O_{\Sigma}(1))=\mathbb C^6.
$$
Since $\mathcal O_{\Sigma}(1)$ is very ample, this would imply that the six independent quadrics vanishing on the rational normal quartic curve in $\mathbb P^4$ also vanish on $\Sigma$.
This is impossible, because $\Sigma$ is a complete intersection of four quadrics.

\medskip$\bullet$
{\it Step 3. Since the map $\varphi: U\otimes V\to W$ is surjective, the sheaves $\mathcal A$ and $\mathcal A^{\vee}(1)$ are locally free and globally generated.}

\smallskip
If $\varphi$ is surjective, then so is
$\mathcal A \otimes \mathcal A^{\vee}\to \mathcal O_{\Sigma}$.
Fix $p\in\Sigma$ and set
\[
R:=\mathcal O_{\Sigma,p},\qquad M:=\mathcal A_p.
\]
Then the map $M\otimes_R M^\vee\to R$
is surjective. This implies the existence of $x\in M$ and $f\in M^\vee$ such that $f(x)=1$.
Therefore, the map
$r\in R\to rx\in M$
has $f:M\to R$ as a retraction. It follows that
$M\simeq xR\oplus\ker f\simeq R\oplus\ker f$.
Since $M$ has rank $1$, the module $\ker f$ has rank zero. Moreover, it is torsion-free, being a submodule of the module $M$,
which is torsion-free. Over an integral domain $R$, a torsion-free module of rank zero is zero. Thus
$M\simeq R$.
Since $p$ is arbitrary, this proves that $\mathcal A$ is locally free of rank $1$. Consequently,
$\mathcal A^{\vee}(1)$ is as well.

From this and the fact that $\varphi$ is surjective, it follows that $\mathcal A$ and $\mathcal A^{\vee}(1)$ are globally
generated. Indeed, suppose for the sake of contradiction that $p$ is a base point of $U=H^0(\Sigma, \mathcal A)$. Since $\mathcal A_p=\mathcal O_{\Sigma, p}$,
we have $u(p)=0$ for every $u\in U$, and thus
$\varphi(u\otimes v)(p)=0$ for every $u\in U, v\in V$.
Since $\varphi$ is surjective, all global sections of $\mathcal O_{\Sigma}(1)$ would vanish at $p$, which is impossible. Therefore, $\mathcal A$ is globally generated. The same argument proves that $\mathcal A^{\vee}(1)$ is globally generated.

\medskip$\bullet$
{\it Step 4. The cases $(x,y)=(2,4)$ and $(x,y)=(4,2)$ are not possible.}

\smallskip
First, suppose $(x,y)=(2,4)$.

\smallskip
Then $U=H^0(\Sigma,\mathcal A)$ is a base-point free pencil. By the base-point free pencil trick,
we have
\[
0\longrightarrow\mathcal A^{-1}
\longrightarrow U\otimes\mathcal O_\Sigma
\longrightarrow\mathcal A\longrightarrow0.
\]
Tensoring by $\mathcal B:=\mathcal A^{\vee}(1)$ yields
\[
0\longrightarrow\mathcal B\otimes\mathcal A^{-1}
\longrightarrow U\otimes\mathcal B
\longrightarrow\mathcal A\otimes\mathcal B\simeq \mathcal O_{\Sigma}(1)
\longrightarrow0.
\]
Passing to global sections, we have
$\ker\varphi\simeq
H^0\bigl(\Sigma,\mathcal B\otimes\mathcal A^{-1}\bigr)$.
On the other hand,
$U\otimes V$ has dimension $8$, and
$\dim W=6$.
The surjectivity of $\varphi$ implies
$\dim\ker\varphi=2$,
and thus
$h^0\bigl(\Sigma,\mathcal B\otimes\mathcal A^{-1}\bigr)=2$.
But $\deg (\mathcal A^{\vee}(1)\otimes \mathcal A^{-1})=8-8=0$. Therefore,
$h^0\bigl(\Sigma,\mathcal B\otimes\mathcal A^{-1}\bigr)\leq 1$,
a contradiction.

\smallskip
Now assume $(x,y)=(4,2)$.

\smallskip
This time, it is $V=H^0(\Sigma,\mathcal B)$ that is a base-point free pencil. We have:
\[
0\longrightarrow\mathcal B^{-1}
\longrightarrow V\otimes\mathcal O_\Sigma
\longrightarrow\mathcal B\longrightarrow0.
\]
Tensoring by $\mathcal A$ we obtain:
\[
0\longrightarrow\mathcal A\otimes\mathcal B^{-1}
\longrightarrow V\otimes\mathcal A
\longrightarrow\mathcal A\otimes\mathcal B\simeq \mathcal O_{\Sigma}(1)
\longrightarrow0.
\]
On global sections we get
$\ker\varphi\simeq
H^0\bigl(\Sigma,\mathcal A\otimes\mathcal B^{-1}\bigr)$.
Since
$\dim(U\otimes V)=8$, and
$\dim W=6$,
surjectivity implies
$h^0\bigl(\Sigma, \mathcal A\otimes\mathcal B^{-1}\bigr)=\dim\ker\varphi=2$.
However, $\deg(\mathcal A\otimes\mathcal B^{-1})=8-8=0$.
Therefore,
$h^0\bigl(\Sigma, \mathcal A\otimes\mathcal B^{-1}\bigr)\leq1$.
Another contradiction. The case $(4,2)$ is also impossible.

\medskip$\bullet$
{\it Step 5. The case $(x,y)=(3,3)$ and $\mathcal A$ isomorphic to $\mathcal A^{\vee}(1)$ is not possible.}

\smallskip
Suppose
$\mathcal A\simeq\mathcal B$.
Since $\mathcal B=\mathcal A^{-1}\otimes \mathcal O_{\Sigma}(1)$, it follows that
$\mathcal A^{\otimes2}\simeq \mathcal O_{\Sigma}(1)$.
After identifying $V$ with $U$, the multiplication becomes
$U\otimes U\longrightarrow H^0(\Sigma,\mathcal A^{\otimes2})\simeq W$.
Due to the commutativity of the product, it factors through
$\operatorname{Sym}^2U\longrightarrow W$.
Since
\[
\dim\operatorname{Sym}^2U=\binom{4}{2}=6=\dim W
\]
and the map is surjective, we have an isomorphism
$\operatorname{Sym}^2U\xrightarrow{\sim}W$.
The morphism defined by $\mathcal A$ is
$\phi_\mathcal A:\Sigma\longrightarrow\mathbb P (U^\vee)\simeq\mathbb P^2$.
If $s_0,s_1,s_2$ is a basis of $U$, the six sections
$s_0^2, s_0s_1, s_0s_2, s_1^2, s_1s_2, s_2^2$
form a basis of $W$. Therefore, the hyperplane morphism factors as
\[
\Sigma\xrightarrow{\phi_\mathcal A}\mathbb P^2
\xrightarrow{\nu_2}\mathbb P^5,
\]
where $\nu_2$ is the Veronese embedding. It follows that
$\Sigma\subseteq\nu_2(\mathbb P^2)$.
The Veronese surface is contained in six linearly independent quadrics.
Thus, $h^0\bigl(\mathbb P^5,\mathcal I_\Sigma(2)\bigr)\geq6$.
This is impossible, because  $\Sigma$ is a complete intersection of four quadrics.
The symmetric $(3,3)$ case is impossible.

\medskip$\bullet$
{\it Step 6. The case $(x,y)=(3,3)$ with $\mathcal A$ not isomorphic to $\mathcal A^{\vee}(1)$ is impossible.}

\smallskip
Suppose now that
$\mathcal A\not\simeq\mathcal B$.
The line bundle $\mathcal B\otimes\mathcal A^{-1}$ has degree zero and is non-trivial. Therefore,
$H^0\bigl(\Sigma,\mathcal B\otimes\mathcal A^{-1}\bigr)=0$.
Consider the Grassmannian of pencils of $U$:
\[
\mathbb G:=\mathbb G(2,U)\simeq\mathbb P(U^\vee)\simeq\mathbb P^2,
\]
and the tautological bundle $\mathcal S\subset U\otimes\mathcal O_\mathbb G$ of rank $2$. Multiplication induces
$$
\Psi:\mathcal S\otimes V\longrightarrow W\otimes\mathcal O_\mathbb G.
$$
Both bundles have rank $6$. Taking the determinant, we get a section
\[
\det\Psi\in H^0(\mathbb G,\mathcal O_\mathbb G(3)).
\]
Its divisor of zeros is therefore a plane cubic, provided the section is not identically zero.
The complete linear system $|\mathcal A|$ defines a map
\[
\phi_\mathcal A:\Sigma\longrightarrow\mathbb P(U^\vee)\simeq \mathbb G.
\]
For $p\in\Sigma$, the plane
$P_p:=\ker\bigl(U\longrightarrow\mathcal A(p)\bigr)$
is the pencil formed by the sections that vanish at $p$. Let
\[
D:=\phi_\mathcal A(\Sigma)\subset \mathbb G\simeq\mathbb P^2.
\]
Then
\[
P\in D\quad\Longleftrightarrow\quad P\text{ has a base point on }\Sigma.
\]
Since $\Sigma$ is integral, $D$ is also an integral curve.

If $P\notin D$, the pencil $P\subset U$ is base-point free. The pencil trick yields
the sequence
\[
0\longrightarrow\mathcal A^{-1}
\longrightarrow P\otimes\mathcal O_\Sigma
\longrightarrow\mathcal A\longrightarrow0.
\]
Tensoring by $\mathcal B$ gives
\[
0\longrightarrow\mathcal B\otimes\mathcal A^{-1}
\longrightarrow P\otimes\mathcal B
\longrightarrow\mathcal A\otimes\mathcal B\simeq \mathcal O_{\Sigma}(1)
\longrightarrow0.
\]
Thus
$
\ker\bigl(P\otimes V\longrightarrow W\bigr)
\simeq H^0\bigl(\Sigma,\mathcal B\otimes\mathcal A^{-1}\bigr)=0$.
Since the domain and codomain both have dimension $6$, the map
$
\Psi_P:P\otimes V\longrightarrow W
$
is an isomorphism. Therefore,
$
\det\Psi(P)\neq0$.
In particular, $\det\Psi$ is not the zero section.
Let  $\operatorname{div}(\det\Psi)$ be
the  divisor of zeros of $\det\Psi$. We have:
$$\operatorname{div}(\det\Psi)\in \left|\mathcal O_{\mathbb G}\left(3\right)\right|.$$

If $P\in D$, there exists $p\in\Sigma$ such that every $u\in P$ vanishes at $p$. For every $v\in V$, the product $uv$ also vanishes at $p$. Thus
$
{\text{Im}}\Psi_P\subseteq H^0\bigl(\Sigma, \mathcal O_{\Sigma}(1)\otimes\mathcal I_p\bigr)
$.
Since $\mathcal O_{\Sigma}(1)$ is globally generated and $\dim W=6$, we have
$
h^0\bigl(\Sigma, \mathcal O_{\Sigma}(1)\otimes\mathcal I_p\bigr)=5$.
Therefore
$
\det\Psi(P)=0
$.

\smallskip
We conclude that
$$\operatorname {Supp}\left(\operatorname{div}(\det\Psi)\right)=D.$$
Therefore, there exists an integer $a\geq1$ such that
$
\operatorname{div}(\det\Psi)=aD$.
As $\det\Psi$ is a section of $\mathcal O_{\mathbb P^2}(3)$,
$
3=a\deg D$.
The image $D$ cannot be a line. Thus
$
\deg D=3$, and
$a=1$.
Let $b$ be the degree of the map $\phi_\mathcal A:\Sigma\to D$. Since $\phi_\mathcal A$ is defined by the complete, base-point free linear system $|\mathcal A|$,
we have $\mathcal A\simeq\phi_\mathcal A^*\mathcal O_D(1)$.
Taking degrees, we get
$\deg\mathcal A=b\deg D$.
But $\deg\mathcal A=8$ and $\deg D=3$. This is absurd.

\smallskip
This concludes the proof of Theorem \ref{2,2,2,2}.
\end{proof}

\section{The proof of Corollary  \ref{general}}

\begin{proof}[Proof of Corollary \ref{general}] As before,
in view of \cite[Corollary 3.3.]{DRS},
we may assume that
$C$ is contained in a hyperplane section $S$ of $X$.
Combining Corollary \ref{stime} with Proposition \ref{m+1} and Corollary \ref{3terms}, which hold true in general when
$X$ is Calabi-Yau,  we get:
$$
p_a(C)\leq \frac{d^2}{2n}+\frac{d}{2}+\rho,
$$
where
$$
\rho= -\frac{1+\epsilon}{2n}(n+1+\epsilon)-p_a(S)+
\sum_{i=m+1}^{m+3}d-h_{\Gamma}(i)
$$
$$
\leq
-\frac{1+\epsilon}{2n}(n+1+\epsilon)-p_a(S)+\frac{n+3+\epsilon}{2}+\frac{3+\epsilon}{2}+1
$$
$$
=\frac{1}{2n}\left[-\epsilon^2+\epsilon(n-2)+(n^2+7n-1)\right]-p_a(S).
$$
An elementary calculation shows that the maximum value of the function
$$
\phi(\epsilon)=\frac{1}{2n}\left[-\epsilon^2+\epsilon(n-2)+(n^2+7n-1)\right]
$$
is obtained for $\epsilon=\frac{n-2}{2}$.
Therefore, by substituting, and taking into account Lemma \ref{acca}
(here we need $h^1(X,\mathcal O_X)=0$), we obtain:
\begin{equation}\label{semifinale}
\rho\leq \frac{5}{8}n+3-p_a(S)=\frac{5}{8}n+4-h^0(X,\mathcal O_X(1)).
\end{equation}
On the other hand, by Kodaira Vanishing Theorem and Riemann-Roch Theorem we have:
\begin{equation*}
h^0(X,\mathcal O_X(1))=\frac{1}{6}n+\frac{1}{12}c_2(X)\cdot S.
\end{equation*}
By Miyaoka inequality \cite[Theorem 1.1]{M} we know that $c_2(X)\cdot S\geq 0$.
Therefore, from (\ref{semifinale}) we get:
$$
\rho\leq \frac{5}{8}n+4-h^0(X,\mathcal O_X(1))\leq \frac{5}{8}n+4-\frac{1}{6}n=
\frac{11}{24}n+4.
$$
\end{proof}


\end{document}